\documentclass{amsart}
\pdfoutput=1

\usepackage{amssymb}
\usepackage{enumitem}
\usepackage{microtype}
\usepackage{url}
\usepackage{graphicx}
\usepackage{tikz}
\usepackage{tikzpagenodes}

\makeatletter
\@namedef{subjclassname@2020}{%
  \textup{2020} Mathematics Subject Classification}
\makeatother

\newtheorem{theorem}{Theorem}[section]
\newtheorem{corollary}[theorem]{Corollary}
\newtheorem{lemma}[theorem]{Lemma}
\newtheorem{proposition}[theorem]{Proposition}

\theoremstyle{definition}
\newtheorem{definition}[theorem]{Definition}

\newtheorem*{xremark}{Remark}

\newtheorem*{xdefinition}{Definition}
\newtheorem*{xtheorem}{Theorem}

\numberwithin{equation}{section}

\DeclareMathOperator{\ord}{ord}
\DeclareMathOperator{\cf}{cf}
\DeclareMathOperator{\Lim}{Lim}

\newcommand{\nullset}{\mbox{\rm \O}}
\renewcommand{\emptyset}{\nullset}
\renewcommand{\setminus}{\!\smallsetminus\!}
\newcommand{\seq}[1]{\langle#1\rangle}
\newcommand{\NN}{\mathbb{N}}

\begin{document}

\title{Compactness and Symmetric Well Orders}
\author[A. Dasgupta]{Abhijit Dasgupta}
\address{University of Detroit Mercy\\
Department of Mathematics\\
Detroit, MI 48221, USA}
\email{dasgupab@udmercy.edu}

\date{2024 Mar 7 (AM for the published
online version, with minor corrections)}

\begin{abstract}
We introduce and investigate a topological form of
St\"ackel's 1907 characterization of finite sets,
with the goal of obtaining an interesting notion that characterizes
usual compactness (or a close variant of it).
Define a \(T_2\) topological space \((X, \tau)\) to be \emph{St\"ackel-compact}
if there is some linear ordering \(\prec\) on \(X\)
such that every non-empty \(\tau\)-closed set contains
a \(\prec\)-least and a \(\prec\)-greatest element.
We find that compact spaces are St\"ackel-compact but not conversely,
and St\"ackel-compact spaces are countably compact.
The equivalence of St\"ackel-compactness with countable compactness
remains open, but our main result is that
this equivalence holds in scattered spaces of
Cantor-Bendixson rank \(< \omega_2\) under ZFC.
Under V=L, the equivalence holds in all scattered spaces.
\end{abstract}

\thanks{This is the accepted manuscript with minor typo corrections
as in the ``Online First'' version
published on 17 January 2024
in the Bulletin of the Polish Academy of Sciences, Mathematics,
DOI: 10.4064/ba230424-28-12 (journal permits posting the
Online First version on archives)}

\maketitle

\begin{tikzpicture}[remember picture,overlay]
 \node[anchor=north west, outer xsep=4em, outer ysep=5ex]
    at (current page.north west)%
    {\fbox{\parbox{3.9in}{Accepted manuscript for
    article DOI: 10.4064/ba230424-28-12\\
    Journal: Bulletin Polish Acad. Sci. Math.\\
    Published online:  17 January 2024\\
    \copyright{}\,Instytut Matematyczny PAN, 2024\\[0.3\baselineskip]
    Journal permits posting the ``Online First'' version on archives.}}};
\end{tikzpicture}

\section{Introduction and Summary of Results.}

A familiar phenomenon in point-set topology is that
a purely combinatorial set-theoretic condition that characterizes
finiteness of sets will often have a corresponding analogue in topological
spaces which characterizes compactness, or at least an interesting
variant of compactness (Tao~\cite{Tao} illustrates this with
example properties;
see also \cite{Engelking}).

The purpose of this article to is to introduce and investigate the topological
analogue of a specific property due to St\"ackel~\cite{Stackel}
that characterizes the finiteness of a set,
namely the existence of some ordering on the set in which
every non-empty subset has
a smallest and a greatest element (a ``symmetric well-order'').
Section~\ref{sec:formulation} defines
the corresponding topological property which we call
\emph{St\"ackel-compactness,}
namely the existence of some ordering on the space such that every non-empty
\emph{closed} subset has a smallest and a greatest element.
Our problem
is to study
how close this notion is to ordinary compactness.

In Section~\ref{sec:propstack}, we establish some basic properties of
St\"ackel-compactness:
We observe that 
every compact Hausdorff space is St\"ackel-compact, but not conversely
(we show that the space \(\omega_1\) is St\"ackel-compact).
Also, every St\"ackel-compact space is countably compact.
Consequently, in metric spaces
St\"ackel-compactness coincides with usual compactness.
Thus the notion has similarities
with other variants of compactness such as
pseudocompactness and sequential compactness
(although distinct from them).  St\"ackel-compactness
appears to be quite close to countable compactness,
but we do not know if the two notions are distinct or if they coincide.

In Section~\ref{sec:ccompact}, we obtain our main result:
\emph{In scattered Hausdorff spaces of
Cantor-Bendixson-rank less than \(\omega_2\),
St\"ackel-compactness coincides with countable compactness}
(Theorem~\ref{theo:main}).

In Section~\ref{sec:bzfc} we go beyond ZFC and
combine our results with those from~\cite{mathoverflow}
to remove the restriction on Cantor-Bendixson rank:
\emph{Under V=L,
all scattered countably compact \(T_2\) spaces are St\"ackel-compact,
so this conclusion is relatively consistent with ZFC}
(Theorem~\ref{theo:veql}).
The referee has pointed out that
there is a class of
scattered spaces, namely Mr\'owka spaces, which provides examples
distinguishing different notions of compactness,
so Theorem~\ref{theo:veql}
tightens the problem of distinguishing
St\"ackel-compact\-ness from countable compactness under ZFC,
and raises the possibility that it (Theorem~\ref{theo:veql})
may hold in ZFC without additional set theoretic axioms.

Section~\ref{sec:open} uses a suggestion from the referee
to show that Nov\'ak's basic method~\cite{Novak} produces
St\"ackel-compact spaces (Proposition~\ref{prop:novakstackel}),
and St\"ackel-compactness is not
a productive property (Corollary~\ref{coro:negprod}).

Finally, in Section~\ref{sec:finish}
we finish with
open questions,
credits,
and history.

\section{Symmetric Well Orders and St\"ackel-compactness.}%
\label{sec:formulation}

\begin{definition}
An order \(\prec\) on a set \(X\) is called a
\emph{symmetric well-order} if every non-empty subset of \(X\)
contains both a least and a greatest element
(or equivalently, if both \(\prec\) and the reverse order
\(\succ\) well order \(X\)).
\end{definition}
\begin{proposition}[St\"ackel~\cite{Stackel}]\label{prop:stackfin}
A set \(X\) is finite if and only if a symmetric well-order can be
defined on \(X\).
\end{proposition}
(This is easily proved in ZF without AC; see \cite{Suppes}, p.108 and p.149.)

\smallskip

From this purely combinatorial
characterization of finite
sets, we now go to its topological analogue: 
We define a topological version of
symmetric well-ordering, and
formulate
our main notion,
\emph{St\"ackel-compactness.} 
\begin{definition}
Let \(X\) be a Hausdorff topological space.  
\begin{enumerate}
\item
An order \(\prec\) on \(X\) is called a \emph{symmetric topological
well-order}~if each non-empty closed subset of \(X\) has
a least and a greatest element.
\item
\(X\) is called \emph{St\"ackel-compact} if there exists
some ordering \(\prec\) on \(X\) which is a symmetric topological well order.
\end{enumerate}
\end{definition}
\emph{Caveat:
In these definitions, the order \(\prec\) on \(X\) 
is not assumed to be related to the
topology on \(X\) in any other way.  In general, the order topology
given by the order \(\prec\) will be
unrelated to
the original topology of \(X\).}

\medskip

The unit interval \([0,1]\) is St\"ackel-compact, as
the usual ordering is itself a symmetric topological well order;
here the topology is same as the order topology.
More generally, a linearly ordered space (a linearly
ordered set under order topology)
is compact if and only if every non-empty closed set has a least
and a greatest element, and so
all compact linearly ordered spaces are St\"ackel-compact.
Conversely, we can ask if the points of a given topological space \(X\)
can be ordered in a way that every non-empty closed set has a least
and a greatest element.   So we can view St\"ackel-compactness
as a natural generalization of compact linearly ordered spaces
(rather than as a topological analogue of finiteness),
which gives an alternative second motivation for our notion.
I thank the referee for this observation.

More examples of St\"ackel-compact spaces will appear later.

In the remaining sections, we focus on our main
problem:

\medskip
\noindent
\textbf{Problem.\;}
How close is St\"ackel-compactness to compactness?
\medskip

Throughout, we restrict our attention to Hausdorff topological
spaces, and assume the Axiom of Choice, i.e., work in ZFC.

(Notations and 
results used
can be found in standard references in point-set topology and set theory,
such as
\cite{Engelking,Jech}.)

\section{Basic Properties of St\"ackel-compact spaces.}%
\label{sec:propstack}
St\"ackel-compactness shares
these properties of ordinary compactness:
\begin{proposition}
A closed subspace of a St\"ackel-compact space is St\"ackel-compact.
If a topology is St\"ackel-compact, then so is any weaker \(T_2\)
topology.
\end{proposition}
\begin{proof}
Immediate from the definition.
\end{proof}
\begin{proposition}\label{prop:compactimpliesSC}
Every compact Hausdorff space \(X\) is St\"ackel-compact.
\end{proposition}
\begin{proof}
\(X\) is homeomorphic to a closed subspace of \([0,1]^\mu\)
for some ordinal \(\mu\).
The lexicographic order on \([0,1]^\mu\)
is readily verified to be a symmetric topological well order.
Hence \([0,1]^\mu\), and so \(X\), is St\"ackel-compact.
\end{proof}
\begin{proposition}
Every St\"ackel-compact
space \(X\) is countably compact.
\end{proposition}
\begin{proof}
Fix an ordering on \(X\) in which every non-empty closed set contains
both a least and a greatest element.
Since \(X\) is Hausdorff,
it suffices to show that every infinite subset of \(X\) has a limit point.
Let \(A\) be an infinite subset of \(X\).  Then
(e.g.\ by Proposition~\ref{prop:stackfin})
\(A\) contains a non-empty
subset \(B\) which either has no least element or has no greatest element.
\(B\) cannot be closed,
since \(X\) is St\"ackel-compact.
Hence \(B\) has
a limit point, and so \(A\) has a limit point.
\end{proof}
\begin{corollary}\label{coro:metricequiv}
Let \(X\) be a Hausdorff space which is either Lindel\"of or
paracompact.
Then \(X\) is St\"ackel-compact if and only if \(X\) is compact.
In particular, in metrizable spaces St\"ackel-compactness coincides with
compactness.
\end{corollary}

At this point, we have the following implications (in Hausdorff spaces):
\[
\text{Compact } \implies \text{ St\"ackel-compact } \implies
\text{Countably compact}
\]

We next ask:  Can the first implication be reversed?
Is every St\"ackel-compact space compact?
It turns out that the answer is negative.

\begin{theorem}\label{theo:omegaone}
The space \(\omega_1\), consisting of all countable ordinals with the
order topology, is St\"ackel-compact.
Thus there are St\"ackel-compact spaces which are not compact.
\end{theorem}
\begin{proof}
Let \(<\) denote the usual ordering on ordinals.
Partition \(\omega_1\) into two stationary sets \(A\) and \(B\),
and let \(\prec\) be the order on \(\omega_1\) in which
``\(A\) is followed by the reverse of \(B\)'', or more precisely
by defining \(\alpha \prec \beta\)
if and only if either \(\alpha \in A\) and \(\beta \in B\),
or \(\alpha, \beta \in A\) and \(\alpha < \beta\),
or \(\alpha, \beta \in B\) and \(\alpha > \beta\).

Let \(F\) be a non-empty closed subset of \(\omega_1\).
We will show that \(F\) contains both a \(\prec\)-least element and
a \(\prec\)-greatest element.

If \(F\) meets both \(A\) and \(B\) then \(F\) clearly has both
a \(\prec\)-least element and a \(\prec\)-greatest element,
so assume that either \(F \subseteq A\) or \(F \subseteq B\).
Then \(F\) must be countable, since an uncountable closed set
in \(\omega_1\) would meet both \(A\) and \(B\).
Suppose that \(F \subseteq A\).
Now \(F\) is a
non-empty countable closed set in \(\omega_1\), and so
\(F\) contains a \(<\)-least and a \(<\)-greatest ordinal
under the usual ordering \(<\) of the ordinals.  But on \(A\),
the ordering \(\prec\) coincides with the usual ordering \(<\)
of the ordinals, so \(F\) contains both a \(\prec\)-least element and
a \(\prec\)-greatest element.
(A similar argument applies if \(F \subseteq B\).)

Thus \(\omega_1\) is St\"ackel-compact.
\end{proof}

A similar argument shows that the
long line
is St\"ackel-compact.
Also,
for general ordinal spaces,
it will follow from the results of this article that
a limit ordinal
\(\lambda\) of uncountable cofinality is St\"ackel-compact
if \(\lambda < \omega_2\), and under V=L all limit ordinals
of uncountable cofinality are St\"ackel-compact.

\medskip

The product of two St\"ackel-compact spaces may not be St\"ackel-compact
(Corollary~\ref{coro:negprod}).  However, if one of the two spaces is
additionally assumed to be compact, then the product will be
St\"ackel-compact.
\begin{proposition}\label{prop:prodcompsc}
Let \(X\) and \(Y\) be Hausdorff spaces.  If \(X\) is St\"ackel-compact
and \(Y\) is compact, then \(X \times Y\) is St\"ackel-compact.
\end{proposition}
\begin{proof}
By Proposition~\ref{prop:compactimpliesSC}, \(Y\) is also St\"ackel-compact.
Therefore we can fix symmetric topological well orders on \(X\) and on \(Y\).
Then the lexicographical order on \(X \times Y\) defined by
\[
(u,v) \prec (x,y)
\iff
u \prec x \text{ in \(X\), or }
u = x \text{ and \(v \prec y\) in \(Y\)}
\]
is a symmetric topological well order on \(X \times Y\).
To see this, let \(C\) be a non-empty closed subset of \(X \times Y\).
Then the projection \(P\) of \(C\) onto \(X\),
\[
P := \{ x \in X \colon (x,y) \in C \text{ for some \(y \in Y\)} \}
\]
is also a non-empty closed subset of \(X\) since \(Y\) is compact.
So \(P\) will have a least element, say \(x_0\), with respect to the
symmetric topological well order on \(X\).  Now the set
\(Q := \{ y \in Y \colon (x_0, y) \in C \}\) is a non-empty closed
subset of \(Y\) and so will have a least element, say \(y_0\),
with respect to the symmetric topological well order on \(Y\).
Then \((x_0, y_0)\) will be the least element of \(C\) under the
lexicographic order.  Similarly \(C\) also has a largest element.
\end{proof}
\begin{corollary}%
\label{coro:prodexample}
The product space \(\omega_1 \times (\omega_1 + 1)\) is St\"ackel-compact.
We thus have St\"ackel-compact Tychonov spaces which are not normal.
(This answers a question of T.\ S.\ S.\ R.\ K.\ Rao from
ICAACA 2020 conference, India.)
\end{corollary}
The basic observations of this section indicate that
St\"ackel-compactness
behaves in ways similar (but not identical) to some other
variants of compactness.  Like pseudo\-compactness
and countable compactness, St\"ackel-compactness
is a necessary but not sufficient condition
for compactness in Hausdorff spaces, and in metric spaces
it coincides with compactness.
Unlike pseudo\-compactness, St\"ackel-compactness implies
countable compactness.
We next look at the question of reversal of this implication.

\section{The Case of Countably Compact Spaces.}
\label{sec:ccompact}
We now ask if the second implication mentioned after
Corollary~\ref{coro:metricequiv} can be reversed:

\smallskip
\noindent
\textbf{Question.\;}
Are all countably compact \(T_2\) spaces St\"ackel-compact?
\smallskip

We do not know the full answer to this question,
but we will prove a partial result and
show that certain types of countably compact spaces
are St\"ackel compact, under certain restrictions.
This is the main result of this article, Theorem~\ref{theo:main}.

When we try to improve Theorem~\ref{theo:main} by
relaxing its restrictive conditions,
we get into set theoretical considerations involving
additional hypotheses
beyond the standard
ZFC axioms
(Section~\ref{sec:bzfc}).

However, in this section
all results are proved under
ZFC.
First, we set up and review some standard terminology and notation.

\begin{definition}
Let \(X\) be a topological space and let \(E\) be a subset of \(X\).
\begin{enumerate}
\item
\(\Lim E\) denotes the set of limit points of \(E\).
\item
\(E\) is \emph{perfect} if
\(E\) is closed and
dense-in-itself,
that is, if \(\Lim E = E\).
\item
The space \(X\) is \emph{scattered}
if no non-empty subset is perfect.
\end{enumerate}
\end{definition}
We get the \emph{Cantor-Bendixson derivatives} of
\(X\)
by repeatedly applying the \(\Lim\) operation through all ordinals,
taking intersections
at limit stages:
\begin{definition}
Let \(X\) be a topological space.
For each ordinal \(\alpha\) we define a subset \(X^{(\alpha)}\)
of \(X\) by transfinite recursion as follows:
\begin{align*}
X^{(0)} &:= X,\\
X^{(\alpha+1)} &:= \Lim X^{(\alpha)},\\
X^{(\alpha)} &:= \bigcap_{\beta < \alpha} X^{(\beta)}\quad
\text{if \(\alpha\) is a limit ordinal.}
\end{align*}
\(X^{(\alpha)}\) is called
\emph{the \(\alpha\)-th Cantor-Bendixson derivative of \(X\).}
\end{definition}
The Cantor-Bendixson derivatives \(X^{(\alpha)}\)
are closed sets that decrease with \(\alpha\),
i.e.\@ \(\alpha < \beta \implies X^{(\alpha)} \supseteq X^{(\beta)}\):
\[
X = X^{(0)}
\supseteq X^{(1)} 
\supseteq X^{(2)} 
\supseteq
\;\cdots\;
\supseteq X^{(\alpha)} \supseteq X^{(\alpha + 1)}
\supseteq
\;\cdots\;,
\]
and there must be an ordinal \(\rho\) with \(X^{(\rho+1)} = X^{(\rho)}\).
\begin{definition}
For a topological space \(X\),
the least ordinal \(\rho = \rho(X)\) such that \(X^{(\rho+1)} = X^{(\rho)}\)
is called the \emph{Cantor-Bendixson rank}
(or \emph{CB-rank}) of \(X\).
\end{definition}
Note that if \(\rho = \rho(X)\) is the Cantor-Bendixson rank of
\(X\), then
\(X\) is perfect if and only if \(\rho = 0\), and
\(X\) is scattered if and only if \(X^{(\rho)} = \emptyset\).

Also,
if \(X\) is countably compact and scattered, then its
CB-rank \(\rho = \rho(X)\) is either a successor ordinal
or must have uncountable cofinality (\(\cf \rho > \omega\)).

The simplest example of a countably compact scattered Hausdorff space
with uncountable CB-rank is \(\omega_1\) (under the usual order topology),
which we saw in Theorem~\ref{theo:omegaone} to be St\"ackel-compact.
We may therefore try to somehow ``lift the proof'' of
Theorem~\ref{theo:omegaone} to general
countably compact scattered Hausdorff spaces.  This is done in
Theorem~\ref{theo:scattered} below under a ``reflection
assumption''.

\medskip

We now state our main result.

\begin{theorem}[ZFC]\label{theo:main}
Let \(X\) be a scattered Hausdorff space with
CB-rank less than \(\omega_2\).
Then
\(X\) is St\"ackel-compact if and only if \(X\) is
countably compact.
\end{theorem}

Using Theorem~\ref{theo:main}, we get more examples
of St\"ackel-compact spaces:
\(X := \omega_1 \times \omega_1\),
\(Y := (\omega_1 + 1) \times (\omega_1 + 1) \setminus
\{ (\omega_1, \omega_1) \} \), and
\(Z := \omega_1 \times (\omega_1 + 1)\).
(In Corollary~\ref{coro:prodexample},
\(Z\) was shown to be St\"ackel-compact 
using Proposition~\ref{prop:prodcompsc}, but
Proposition~\ref{prop:prodcompsc} does not help
for \(X\) or \(Y\).)

\medskip

The rest of the section is for the proof of Theorem~\ref{theo:main}.

\begin{definition}
Let \(A\) be a set and let \(\alpha\) be an ordinal with \(\cf \alpha > \omega\).
We say that:
\begin{enumerate}
\item
\emph{\(A\) reflects at \(\alpha\)} if \(A \cap \alpha\) is stationary in \(\alpha\).
\item
\emph{\(A\) reflects everywhere on \(\rho\)} (where \(\rho\) is an
arbitrary ordinal) if
\(A\) reflects at \(\beta\)
for every \(\beta \leq \rho\) with
\(\cf \beta > \omega\).
\end{enumerate}
\end{definition}

\begin{theorem}[ZFC]\label{theo:scattered}
Let \(X\) be a countably compact scattered Hausdorff space with
CB-rank \(\rho = \rho(X)\), and suppose that
there exist disjoint sets \(A\) and \(B\) each of which reflects
everywhere on \(\rho\).  Then \(X\) is St\"ackel-compact.
\end{theorem}
\begin{proof}
The proof
improves upon
the proof that \(\omega_1\) is
St\"ackel-compact.

By the given condition, we can partition \(\rho\) into two disjoint sets
\(A, B\) such that each of \(A\) and \(B\) reflects everywhere on \(\rho\). 

Define, for each ordinal \(\alpha\):
\[
Y_\alpha := X^{(\alpha)} \setminus X^{(\alpha + 1)}.
\]
Then \(\{ Y_\alpha \,\colon\; \alpha < \rho \}\) forms a partition of \(X\).

Fix a well order of \(X\) such that \(Y_\alpha\) precedes \(Y_\beta\)
in this order if \(\alpha < \beta < \rho\).

Now define:
\[
X_A := \bigcup_{\alpha \in A} Y_\alpha
\qquad
\text{ and }
\qquad
X_B := \bigcup_{\beta \in B} Y_\beta.
\]
Then \(X_A\) and \(X_B\) form a partition of \(X\).

Now take the order on \(X\) in which \(X_A\) precedes \(X_B\),
\(X_A\) is ordered by the above well-order,
and \(X_B\) is ordered by the reverse of that well-order.

We now show that under this order, every non-empty closed
subset \(F\) of \(X\) has a least and a greatest element.

Given a non-empty closed set \(F\) in \(X\), consider the two sets
\(F \cap X_A\) and  \(F \cap X_B\).  If both of these sets are
non-empty, then \(F\) will contain least and greatest elements
(since \(X_A\) is well-ordered and \(X_B\) is reverse well-ordered
by our new chosen order on \(X\)).  So we may assume that one of
the sets \(F \cap X_A\) and \(F \cap X_B\) is empty, and without
loss of generality that \(F \cap X_B = \emptyset\), that is,
\(F \subseteq X_A\).  So \(F\) has a least element. We will show
that \(F\) has a greatest element as well.

Define:
\[
C := \{ \alpha < \rho \,\colon\; F \cap Y_\alpha \neq \emptyset \}.
\]
Then \(C \subseteq A\) by our assumption that \(F \subseteq X_A\), and
\(C \neq \emptyset\) since \(F \neq \emptyset\).

If \(C\) has a largest element \(\mu\), then \(F \cap Y_\mu\) must be
non-empty finite by countable compactness, and so will have a largest
element, which must then be the greatest element of \(F\).  Hence
it suffices to show that \(C\) has a largest element.

Suppose (for contradiction) that \(C\) does not have a largest element.
Then \(\sup C \in \Lim C \,\setminus\, C\).  Let:
\[
\mu := \min ( \Lim C \,\setminus\, C).
\]
Thus \(\mu\) is a limit ordinal \(\leq \rho\),
and \(\cf \mu \geq \omega_1\) by countable
compactness.  Note that \(C \cap \mu\) is a
closed unbounded set in \(\mu\).  Now,
since \(B\) reflects at \(\mu\), \(B \cap \mu\) is stationary in \(\mu\),
and so \((C \cap \mu) \cap (B \cap \mu)\) must be non-empty.
But this implies that \(C \cap B \neq \emptyset\) which is a contradiction
since \(C \subseteq A\).
\end{proof}
Our goal now is to try to use the above theorem to show that
if a scattered Hausdorff space is countably compact, then
it is St\"ackel-compact.
But, as mentioned earlier,
we are unable to do this without
additional set-theoretic hypothesis beyond ZFC
(Section~\ref{sec:bzfc}).
In ZFC alone,
we can
use the next theorem below along with Theorem~\ref{theo:scattered}
to obtain the result for spaces with CB-rank \(< \omega_2\),
giving us Theorem~\ref{theo:main}.
\begin{theorem}[ZFC]%
\label{theo:lessthanomega2}
If \(\rho < \omega_2\), then there exist disjoint sets
\(A\) and \(B\) each of which reflects everywhere on \(\rho\).
\end{theorem}

\begin{proof}
The proof is by induction on \(\rho\).

Suppose that \(\rho < \omega_2\) and that for every \(\xi < \rho\)
there are disjoint \(A_\xi, B_\xi\) each of which
reflects everywhere on \(\xi\).

Without loss of generality we can assume that \(\rho\) is a limit
ordinal, so there are two cases:  \(\cf \rho = \omega\) and
\(\cf \rho = \omega_1\).

\medskip

\textbf{Case 1}: \(\cf \rho = \omega\).  We can then choose
a countable sequence of ordinals
\[
0 = \rho_0  < \rho_1 < \dots < \rho_n < \rho_{n+1} < \dots
\]
such that \(\sup_n \rho_n = \rho\).  By induction hypothesis,
for each \(n \in \omega\) we can fix disjoint sets \(A_n, B_n\)
such that both of them reflect everywhere on \(\rho_{n+1}\).
Let:
\[
A := \bigcup_{n \in \omega} A_n \cap (\rho_{n+1} \setminus \rho_n)
\qquad \text{ and } \qquad
B := \bigcup_{n \in \omega} B_n \cap (\rho_{n+1} \setminus \rho_n).
\]
(For ordinals \(\alpha\) and \(\beta\), the set-difference
\(\alpha \setminus \beta\) equals
\(\{\xi \,\colon\; \beta \leq \xi < \alpha \}\).)

Notice that \(A \cap B = \emptyset\).  We show that both \(A\) and \(B\)
reflect everywhere on \(\rho\).  Suppose that \(\alpha \leq \rho\),
with \(\cf \alpha \geq \omega_1\).  Then \(0 < \alpha < \rho\) (since
\(\cf \rho = \omega\) and \(\cf \alpha \geq \omega_1\)), and so
there is \(n\) such that \(\rho_n < \alpha \leq \rho_{n+1}\).
Now \(A_n\) reflects everywhere on \(\rho_{n+1}\), so
\(A_n \cap \alpha\) is stationary in \(\alpha\), and therefore
\(A_n \cap (\alpha \setminus \rho_n)\) is also
stationary in \(\alpha\)
(as \(\alpha \setminus \rho_n\) is closed unbounded in \(\alpha\)).
But
\[
A_n \cap (\alpha \setminus \rho_n) \subseteq
A_n \cap (\rho_{n+1} \setminus \rho_n)
\subseteq A,
\]
hence \(A \cap \alpha\) is stationary in \(\alpha\).
Similarly, \(B \cap \alpha\) is stationary in \(\alpha\).
Thus both \(A\) and \(B\) reflect everywhere on \(\rho\).

\textbf{Case 2}: \(\cf \rho = \omega_1\).  Fix \(E \subseteq \rho\)
of order type \(\omega_1\) with \(\sup E = \rho\), and let
\(L := \rho \,\cap\, \Lim E\).
Then \(L\) is closed unbounded
in \(\rho\) of order type \(\omega_1\) and each \(\lambda \in L\)
is a limit ordinal of countable cofinality.
Enumerate \(L\) increasingly as
\(\seq{\lambda_\xi}_{\xi < \omega_1}\):
\[
L = \{ \lambda_\xi \,\colon\; 0 \leq \xi < \omega_1 \},
\quad\text{with \(\lambda_\alpha < \lambda_\beta\) for all
    \(\alpha < \beta < \omega_1\)}.
\]
Now, for each \(\alpha < \omega_1\), we have \(\lambda_{\alpha+1} < \rho\),
so by induction hypothesis we can find disjoint sets
\(A_\alpha, B_\alpha\) such that both
\(A_\alpha\) and \(B_\alpha\) reflect everywhere on
\(\lambda_{\alpha + 1}\).
Define:
\begin{align*}
A^* &:= 
    \bigcup_{\alpha < \omega_1} \left[ A_\alpha \cap
        (\lambda_{\alpha + 1} \setminus (\lambda_{\alpha} + 1)) \right],\\
B^* &:= 
    \bigcup_{\alpha < \omega_1} \left[ B_\alpha \cap
        (\lambda_{\alpha + 1} \setminus (\lambda_{\alpha} + 1)) \right].
\end{align*}
Note that the sets \(A^*\), \(B^*\), and \(L\) are pairwise disjoint.
As \(L\) is
closed unbounded
in \(\rho\), we can fix disjoint subsets
\(C\) and \(D\) of \(L\) that are stationary in $\rho$.  Finally,
define:
\[
A := A^* \cup C \qquad\text{and}\qquad B := B^* \cup D.
\]
Then \(A \cap B = \emptyset\).
We show that both \(A\) and \(B\)
reflect everywhere on \(\rho\).  Suppose that \(\alpha \leq \rho\),
with \(\cf \alpha > \omega\).
If \(\alpha = \rho\), then
\(C\), and so \(A\),
is stationary in \(\rho = \alpha\).
If \(\alpha < \rho\), then \(\lambda_{\xi} \leq \alpha < \lambda_{\xi + 1}\)
for some \(\xi < \omega_1\).  Since the limit ordinal \(\lambda_{\xi}\)
has countable cofinality but \(\cf \alpha > \omega\), we get
\(\lambda_{\xi} < \alpha\).  Also, \(A_\xi\) reflects everywhere on
\(\lambda_{\xi + 1}\), so \(A_\xi \cap \alpha\) is stationary in \(\alpha\),
and therefore \(A_\xi \cap (\alpha \setminus (\lambda_{\xi}+1))\)
is stationary in \(\alpha\)
(as \(\alpha \setminus (\lambda_{\xi}+1)\) is closed unbounded in \(\alpha\)).
But
\[
A_\xi \cap (\alpha \setminus (\lambda_{\xi}+1))
\subseteq
A_\xi \cap (\lambda_{\xi + 1} \setminus (\lambda_{\xi}+1))
\subseteq
A^*
\subseteq
A,
\]
hence \(A \cap \alpha\) is stationary in \(\alpha\).
Similarly, \(B \cap \alpha\) is stationary in \(\alpha\).
Thus both \(A\) and \(B\) reflect everywhere on \(\rho\).
\end{proof}

Theorem~\ref{theo:main} now follows immediately from
Theorem~\ref{theo:scattered} and Theorem~\ref{theo:lessthanomega2}.

\medskip

\section{CB-rank \(\omega_2\) and beyond.}%
\label{sec:bzfc}
This section improves Theorem~\ref{theo:main} using extra
set-theoretic hypotheses that are relatively consistent with ZFC.
We first show that under \(\square_{\omega_1}\), every
countably compact scattered Hausdorff space of CB-rank \(\omega_2\)
is St\"ackel-compact (Corollary~\ref{coro:myfinal}).
Then, combining our Theorem~\ref{theo:scattered} with a
result of Hamkins~\cite{mathoverflow},
we get:
\emph{Under ZFC\(+\)V=L,
St\"ackel-compactness coincides with countable compactness
in all scattered Hausdorff spaces}
(Theorem~\ref{theo:veql}).

We use set-theoretic terminology from Jech \cite{Jech}.
For any well-ordered set \(X\), let
\(\ord(X)\) denote
the unique ordinal order-isomorphic to \(X\)
(the order-type of \(X\)).
ZFC\(+\)V=L denotes the axioms of ZFC augmented with
G\"odel's Axiom of Constructibility V=L (which is
consistent relative to ZFC).

Let \(\kappa\) be an uncountable cardinal.
Jensen's \emph{square principle} \(\square_\kappa\)
is the following statement,
true under ZFC\(+\)V=L
(see Jech \cite{Jech}):

\smallskip

\noindent
\(\square_\kappa\):
There is a sequence
\(\seq{C_\alpha \colon \alpha < \kappa^+, \text{\(\alpha\) a limit ordinal}}\)
of sets, known as a \emph{\(\square_\kappa\)-sequence,}
such that for all limit \(\alpha < \kappa^+\) we have:
\begin{enumerate}
\item
\(C_\alpha\) is closed unbounded in \(\alpha\);
\item
\(\cf \alpha < \kappa\) implies
that \(C_\alpha\) has cardinality less than \(\kappa\); and
\item
\(\beta \in \Lim(C_\alpha)\) implies \(C_\beta = C_\alpha \cap \beta\).
\end{enumerate}
For such a \(\square_\kappa\)-sequence, we have
\(\ord(C_\alpha) \leq \kappa\) for all limit \(\alpha < \kappa^+\).
\begin{proposition}\label{prop:omegatwo}
Assume \(\square_{\omega_1}\).  Then \(\omega_2\) has a pair of disjoint
subsets which reflect everywhere on \(\omega_2\).
\end{proposition}
\begin{proof}
By \(\square_{\omega_1}\), we can fix a sequence
\(\seq{C_\alpha \colon \alpha < \omega_2, \text{\(\alpha\) a limit ordinal}}\)
such that for all limit \(\alpha < \omega_2\):
\begin{enumerate}
\item
\(C_\alpha\) is closed unbounded in \(\alpha\);
\item
\(\cf \alpha = \omega_1\) implies \(\ord(C_\alpha) = \omega_1\); and
\item
\(\beta \in \Lim(C_\alpha)\) implies \(C_\beta = C_\alpha \cap \beta\).
\end{enumerate}
Fix disjoint stationary subsets \(A_0, B_0\) of \(\omega_1\)
consisting of limit ordinals.
For each \(\mu < \omega_2\) with \(\cf \mu = \omega_1\),
the set \(C_\mu\) is closed unbounded in \(\mu\) and
of order type \(\omega_1\), so we can
form ``isomorphic copies of \(A_0\) and \(B_0\) within \(C_\mu\)''
(under the unique order-isomorphism
between \(\omega_1\) and \(C_\mu\)) by defining:
\[
A_\mu := \{ \xi \in C_\mu \colon \ord(C_\mu \cap \xi) \in A_0 \}
\;\text{ and }\;
B_\mu := \{ \xi \in C_\mu \colon \ord(C_\mu \cap \xi) \in B_0 \}.
\]
Note that \(A_\mu \subseteq \Lim C_\mu\) and \(B_\mu \subseteq \Lim C_\mu\).
Finally, ``put them all together'' by defining:
\[
A := \bigcup_{\substack{\mu < \omega_2 \\ \cf \mu = \omega_1}} A_\mu
\qquad\text{ and }\qquad
B := \bigcup_{\substack{\mu < \omega_2 \\ \cf \mu = \omega_1}} B_\mu.
\]
Then \(A\) and \(B\) are disjoint, since if \(\xi \in A_\mu \cap B_\nu\)
then \(\xi \in \Lim C_\mu \cap \Lim C_\nu\), so
\(C_\mu \cap \xi = C_\xi = C_\nu \cap \xi\),
so \(\ord(C_\mu \cap \xi) = \ord(C_\nu \cap \xi)\), which is a
contradiction since \(\ord(C_\mu \cap \xi) \in A_0\)
and \(\ord(C_\nu \cap \xi) \in B_0\), while \(A_0\) and \(B_0\) are disjoint.

Now if \(\mu < \omega_2\) and \(\cf \mu = \omega_1\), then
\(A_\mu\) and \(B_\mu\) are stationary in \(C_\mu\) and hence in \(\mu\),
so \(A\) and \(B\) reflect at \(\mu\).  Also,
\(A\) and \(B\) reflect at \(\mu = \omega_2\) as well, since
they
are stationary in \(\omega_2\).
So \(A\) and \(B\) are disjoint sets which reflect everywhere on \(\omega_2\).
\end{proof}
Combining the above proposition with Theorem~\ref{theo:scattered} we get:
\begin{corollary}\label{coro:myfinal}
Assuming \(\square_{\omega_1}\), every countably compact scattered
Hausdorff space of CB-rank \(\omega_2\) is St\"ackel-compact.
\end{corollary}
Hamkins (MathOverflow post~\cite{mathoverflow}, enclosed below)
shows that if the global square principle is assumed, then
for every \(\rho \geq \omega_2\) there are disjoint
stationary sets which reflect everywhere on \(\rho\)
(in fact, there exist
disjoint proper classes \(A, B\) of ordinals such that \(A \cap \alpha\)
and \(B \cap \alpha\) are stationary in \(\alpha\) for every ordinal
\(\alpha\) with uncountable cofinality).
Since the global square principle holds under ZFC\(+\)V=L, we can
combine the result of Hamkins with Theorem~\ref{theo:scattered}
to get the following conclusion.
\begin{theorem}\label{theo:veql}
Assume ZFC\(+\)V=L, and let \(X\) be a scattered Hausdorff space.
Then \(X\) is St\"ackel-compact if and only if it is
countably compact.  Hence it is relatively consistent with ZFC that 
all scattered countably compact Hausdorff spaces are St\"ackel-compact.
\end{theorem}
As mentioned earlier, 
while Theorem~\ref{theo:veql}
is restricted to scattered spaces, Mr\'owka spaces (see~\cite{Engelking})
provide examples
of scattered spaces which distinguish different notions of compactness.
So this limits the types of spaces in which a non-St\"ackel-compact
countably compact space may be found under ZFC, further raising
the possibility that
Theorem~\ref{theo:veql} may hold in ZFC (without
assuming V=L).  I thank the referee for this observation.

At the suggestion of the referee, we end this section by
enclosing the relevant parts
from~\cite{mathoverflow} here for completeness.
\begin{xdefinition}[see \cite{Schimmerling}, Definition~19]
The global square principle \(\square\)
is the assertion that there is an assignment \(\nu\mapsto C_\nu\)
for all singular ordinals \(\nu\), such that
\begin{itemize}
\item
\(C_\nu\) is a closed subset of \(\nu\), containing only
        singular ordinals;
\item
if \(\nu\) has uncountable cofinality, then \(C_\nu\) is
unbounded in \(\nu\);
\item
the order type of \(C_\nu\) is less than \(\nu\);
\item
and if \(\mu\in C_\nu\), then \(C_\mu = C_\nu \cap \mu\).
\end{itemize}
\end{xdefinition}
\begin{xtheorem}[Hamkins~\cite{mathoverflow}]\label{prop:hamkins}
Under the global square principle \(\square\),
there is a global partition of the class of
singular ordinals into \(A \sqcup B\), such that
for every \(\kappa\) of uncountable cofinality, both
\(A \cap \kappa\) and \(B \cap \kappa\) are stationary in \(\kappa\).
\end{xtheorem}
\begin{proof}[Proof (Hamkins, reproduced from~\cite{mathoverflow})]
Fix the \(\square\) sequence \(C_\nu\). First, define
    \(A\) and \(B\) up to \(\omega_1\) to be any partition of the
    singular countable ordinals into stationary sets.  Suppose now
    that \(A\) and \(B\) are defined up to \(\nu\), a singular limit
    ordinal.  Consider \(C_\nu\), which has some order type
    \(\eta < \nu\). If \(\eta \in A\), then put \(\nu \in A\),
    otherwise, put \(\nu \in B\). Continue by transfinite recursion.
    Note that \(A\) and \(B\) partition the singular ordinals.

Suppose that \(\kappa\) has uncountable cofinality. If
    \(\kappa = \omega_1\), then \(A \cap \kappa\) and \(B \cap \kappa\)
    are the stationary sets that we used to start the construction.
    More generally, if \(\kappa > \omega_1\) but has cofinality
    \(\omega_1\), then \(\kappa\) is singular and so \(C_\kappa\)
    is a club of some type \(\beta < \kappa\). Further,
    \(A\) and \(B\) when restricted to \(C_\kappa\) are copies of
    \(A \cap \beta\) and \(B \cap \beta\), which by induction
    are each stationary.  So \(A \cap \kappa\) and \(B \cap \kappa\)
    are stationary.  Finally, we have the case that \(\kappa\) has
    cofinality larger than \(\omega_1\).  Fix any club
    \(C \subset \kappa\). Thus, there is some singular
    \(\eta \in C\) with uncountable cofinality. So
    \(C_\eta \cap C\) is club in \(\eta\) and thus meets both
    \(A\) and \(B\), and so \(C\) meets both \(A\) and \(B\),
    as desired.
\end{proof}

\section{St\"ackel-compactness and Nov\'ak Spaces}%
\label{sec:open}
I am greatly indebted to the referee for the results of this section,
as they were obtained after the referee
suggested
that Nov\'ak spaces may resolve the question of productivity
of St\"ackel-compact spaces,
a question that
was left open in the original
version of the article.
This was indeed the case and led readily to a negative answer
to the question (Corollary~\ref{coro:negprod}).
Moreover, Nov\'ak's
method produces
examples of non-compact St\"ackel-compact spaces
without using the theory of stationary sets
(Proposition~\ref{prop:novakstackel});
all our earlier such examples needed stationary sets.

\smallskip

Let \(\beta(\NN)\) denote the Stone-\v{C}ech compactification of the
discrete space \(\NN\) of positive integers.
For any set \(E\), let \(|E|\) denote its cardinality.

\begin{definition}\label{def:novak}
We will call
a subspace \(X\) of \(\beta(\NN)\) a
\emph{basic Nov\'ak space} if
\(X = \bigcup_{\xi < \omega_1} X_\xi\)
for some \(\omega_1\)-sequence \(\seq{X_\xi}_{\xi < \omega_1}\)
such that 
for all \(\xi < \omega_1\),
\(\aleph_0 \leq |X_\xi| \leq 2^{\aleph_0}\), \(X_\xi\) is disjoint from
\(\bigcup_{\mu < \xi} X_\mu\),
and every infinite subset of
\(\bigcup_{\mu < \xi} X_\mu\) 
has a limit point in \(X_\xi\).
\end{definition}
\begin{corollary}[Immediate from Definition~\ref{def:novak}]%
\label{coro:novaksubs}
Let \(X = \bigcup_{\xi < \omega_1} X_\xi\)
be a basic Nov\'ak space with  \(\seq{X_\xi}_{\xi < \omega_1}\)
as in Definition~\ref{def:novak}.

\vspace{0.5ex}

\textup{(a)}~For any infinite
\(S \subseteq X\) there is \(\xi < \omega_1\) such that \(S\)
has limit points in \(X_\nu\) for all \(\nu \geq \xi\).

\textup{(b)}~If \(I \subseteq \omega_1\) is uncountable, then
\(\bigcup_{\xi \in I} X_\xi\) is also a basic Nov\'ak space.
\end{corollary}
\begin{proposition}[Nov\'ak]\label{prop:novak}
If \(A \subseteq \beta(\NN)\)
and \(\aleph_0 \leq |A| \leq 2^{\aleph_0}\),
there is a basic Nov\'ak space \(X = \bigcup_{\xi < \omega_1} X_\xi\)
with \(X_0 = A\) and \(\seq{X_\xi}_{\xi < \omega_1}\) as in
Definition~\ref{def:novak}.
There are basic Nov\'ak spaces \(Y, Z\) such that
\(Y \!\times\! Z\) is not countably compact.
\end{proposition}
\begin{proof}[Proof
(outline, from Nov\'ak's construction~\cite{Novak})]
To get \(X\),
take \(X_0 = A\) and
define via transfinite recursion the sets \(X_\xi\),
which can be chosen to meet
the conditions of Definition~\ref{def:novak}
because
any infinite \(S \subseteq \beta(\NN)\) has \(2^{2^{\aleph_0}}\)
limit points, and if \(|S| \leq 2^{\aleph_0}\) then
\(|\{ E \subseteq S \mid |E| \leq \aleph_0\}| \leq 2^{\aleph_0}\).
Now take such a basic Nov\'ak space \(X\) with \(X_0 = \NN\),
fix uncountable disjoint \(I, J \subseteq \omega_1 \smallsetminus \{0\}\),
and let \(Y := \bigcup_{\xi \in I \cup \{0\}} X_\xi\) and
\(Z := \bigcup_{\xi \in J \cup \{0\}} X_\xi\).
This gives basic Nov\'ak spaces
\(Y\) and \(Z\)
with \(Y \cap Z = \NN\).
Then \(Y \times Z\) is not countably compact,
since it has an infinite discrete closed subset
\(\{(n,n)\mid n\in \NN\}\).
\end{proof}
\begin{lemma}\label{lemma:disjointlimitsets}
Let \(X\) be a Hausdorff topological space
containing two disjoint sets \(A\) and \(B\)
such that every infinite subset of \(X\) has a limit point
in \(A\) and also a limit point in \(B\).
Then \(X\) is St\"ackel-compact.
\end{lemma}
\begin{proof}
Fix a linear order on \(X\) in which
the set \(A\) wholly precedes the set \(X \smallsetminus A\),
the set \(A\) is well ordered, and the set
\(X \smallsetminus A\) is reverse well ordered.
\end{proof}
\begin{proposition}\label{prop:novakstackel}
Every basic Nov\'ak space \(X\) is St\"ackel-compact.
\end{proposition}
\begin{proof}
Let \(X\) be a basic Nov\'ak space and express it as
\(X = \bigcup_{\xi < \omega_1} X_\xi\) with
\(\seq{X_\xi}_{\xi < \omega_1}\) as in Definition~\ref{def:novak}.
Fix uncountable disjoint \(I, J \subseteq \omega_1\), and let
\( A := \bigcup_{\xi \in I} X_{\xi}\) and
\( B := \bigcup_{\xi \in J} X_{\xi}\).
The sets \(A, B\) are disjoint, and
every infinite subset of \(X\) has limit points
in \(A\) and also in \(B\)
(Corollary~\ref{coro:novaksubs}(a)).
So \(X\) is St\"ackel-compact by
Lemma~\ref{lemma:disjointlimitsets}.
\end{proof}
\begin{corollary}\label{coro:negprod}
There are St\"ackel-compact spaces \(Y, Z\) such that
\(Y \times Z\) is not countably compact,
so St\"ackel-compactness is not a productive property.
\end{corollary}
\begin{xremark}
A basic Nov\'ak space \(X = \bigcup_{\xi < \omega_1} X_\xi\)
is a subspace of \(\beta(\NN)\) of size \(2^{\aleph_0}\),
and gives
an example of St\"ackel-compact space that is neither compact
nor sequentially compact (unlike \(\omega_1\)).
\end{xremark}

\section{Open Questions, Credits, and History}%
\label{sec:finish}
\subsection*{Open Questions.}
We list some problems unanswered in this article.
\begin{enumerate}
\item\label{enumi:main}
If \(X\) is a countably compact Hausdorff space, is \(X\) St\"ackel-compact?
\emph{This is the most significant question we have left unsettled.}
\item
Can we answer~(\ref{enumi:main}) if we also assume that
\(X\) is a Tychonov space?
\item
Is every St\"ackel-compact space (\(T_2\) by definition) regular?
\item
Is the continuous image of a St\"ackel-compact space
in a Hausdorff space necessarily St\"ackel-compact?
\end{enumerate}

\subsection*{Acknowledgements and History}
The author initially
obtained the results up to Corollary~\ref{coro:myfinal}
and presented them in seminars in Detroit and Ann Arbor.
Ioannis Souldatos then posted
a MathOverflow question in~\cite{mathoverflow}.
The answers there by Joel Hamkins and Todd Eisworth showed that
the existence of disjoint sets which reflect everywhere is
both consistent with and (modulo large cardinals) independent of ZFC.
Finally, the author thanks the anonymous referee for many valuable
suggestions (mentioned at several places above), the most substantial
of which, on Nov\'ak spaces, resulted in new examples of
non-compact St\"ackel-compact spaces (Proposition~\ref{prop:novakstackel})
and the resolution of an open question in the original version of the article
(Corollary~\ref{coro:negprod}).

\end{document}